\documentclass[12pt,reqno]{amsart}
\usepackage{amsmath,amssymb,cite,geometry,bm,easybmat,enumerate}
\usepackage[colorlinks,citecolor=blue,linkcolor=blue,urlcolor=blue]{hyperref}
\newtheorem{theorem}{Theorem}[section]

\newtheorem{proposition}[theorem]{Proposition}

\numberwithin{equation}{section}
\allowdisplaybreaks[4]
\date{\today}

\begin{document}

\title[Boundary determination of the metric from Cauchy data]{Boundary determination of the Riemannian metric from Cauchy data for the Stokes equations}

\author{Xiaoming Tan}
\address{Beijing International Center for Mathematical Research, Peking University, Beijing 100871, China}
\email{tanxm@pku.edu.cn}


\keywords{Boundary determination; Riemannian metric; Stokes equations; Dirichlet-to-Neumann map; Cauchy data.\\
{\bf-----------------}\\
\hspace*{3mm} {\it Email address}: tanxm@pku.edu.cn\\
\hspace*{3mm} Beijing International Center for Mathematical Research, Peking University, Beijing 100871, China}

\begin{abstract}
    For a compact connected Riemannian manifold of dimension $n$ with smooth boundary, $n\geqslant 2$, we prove that the Cauchy data (or the Dirichlet-to-Neumann map) for the Stokes equations uniquely determines the partial derivatives of all orders of the metric on the boundary of the manifold.
\end{abstract}

\maketitle 

\section{Introduction}

\vspace{5mm}

\subsection{Stokes equations on the manifold}

Let $(M,g)$ be a compact connected Riemannian manifold of dimension $n$ with smooth boundary $\partial M$, $n\geqslant 2$. In this paper, we assume that $M$ is filled with an incompressible fluid. In the local coordinates $\{x_j\}_{j=1}^n$, we denote by $\bigl\{\frac{\partial}{\partial x_j}\bigr\}_{j=1}^n$ and $\{dx_j\}_{j=1}^n$, respectively, the natural basis for the tangent space $T_x M$ and the cotangent space $T_x^{*} M$ at the point $x\in M$. In what follows, we will use the Einstein summation convention. The Greek indices run from 1 to $n-1$, whereas the Roman indices run from 1 to $n$, unless otherwise specified. Then, the Riemannian metric $g$ is given by $g = g_{jk} \,dx_j\otimes dx_k$. Let $\nabla_j=\nabla_{\frac{\partial}{\partial x_j}}$ be the covariant derivative with respect to $\frac{\partial}{\partial x_j}$ and $\nabla^j= g^{jk} \nabla_k$, where $[g^{jk}]=[g_{jk}]^{-1}$. 

Let the smooth vector field $\bm{u}=u^j\frac{\partial}{\partial x_j}$ be the velocity of the fluid. The strain tensor $S$ is defined by (see \cite[p.\,562]{Taylor11.3})
\begin{align}\label{1.1}
    (S\bm{u})^j_k := \nabla^j u_k + \nabla_k u^j,
\end{align}
where $u_k=g_{kl}u^l$, or equivalently, $(S\bm{u})^{jk} := \nabla^j u^k + \nabla^k u^j$. The stress tensor $\sigma$ is given by
\begin{align}\label{1.2}
    \sigma(\bm{u},p):=\mu S\bm{u}-pg,
\end{align}
where $\mu,p\in C^{\infty}(M)$ denote the viscosity and the pressure, respectively. Physically, the case of $\mu=0$ is observed only in superfluids that have the ability to self-propel and travel in a way that defies the forces of gravity and surface tension. Otherwise the viscosities of all fluids are positive. Thus, we can assume that $\mu>0$ in $M$. A fluid with nonconstant viscosity is called a non-Newtonian fluid, and these are relatively common, including things such as blood, shampoo and custard (see \cite{LaiUhlmWang15}). The famous stationary Stokes equations on the manifold read
\begin{align}\label{1.3}
    \begin{cases}
        \operatorname{div}\sigma(\bm{u},p)=0 &\text{in}\ M,\\
        \operatorname{div}\bm{u}=0 \quad & \text{in}\ M,
    \end{cases}
\end{align}
where $\operatorname{div}$ denotes the divergence operator on the manifold.

\subsection{Dirichlet-to-Neumann map}

Let $\bm{h}\in [H^{3/2}(\partial M)]^n$ satisfy the compatibility condition
\begin{align*}
    \int_{\partial M} g(\bm{h},\nu) \,dS=0,
\end{align*}
where $\nu$ is the unit outer normal to $\partial M$. This condition leads to the uniqueness of \eqref{1.3} (see \cite{HeckLiWang07,LaiUhlmWang15}), that is, there exists a unique solution $(\bm{u},p)\in [H^2(M)]^n\times H^1(M)$ ($p$ is unique up to a constant) of the Dirichlet problem
\begin{align}\label{1.4}
    \begin{cases}
        \operatorname{div}\sigma(\bm{u},p)=0 &\text{in}\ M,\\
        \operatorname{div}\bm{u}=0 \quad & \text{in}\ M,\\
        \bm{u}=\bm{h} \quad & \text{on}\ \partial M.
    \end{cases}
\end{align}
We could define the Cauchy data for the Stokes equations by 
\begin{align}
    C_g:=\{(\bm{u},\sigma(\bm{u},p)\nu)|_{\partial M}:(\bm{u},p)\ \text{satisfies \eqref{1.4}}\}.
\end{align}
The physical sense of $\sigma(\bm{u},p)\nu|_{\partial M}$ is the stress acting on $\partial M$ and is called the Cauchy force (see \cite{HeckLiWang07,LaiUhlmWang15}). We also call $\sigma(\bm{u},p)\nu|_{\partial M}$ the Neumann boundary condition for \eqref{1.4}. Thus, we can define the Dirichlet-to-Neumann map $\Lambda_g:[H^{3/2}(\partial M)]^n\to [H^{1/2}(\partial M)]^n$ associated with \eqref{1.4} by 
\begin{align}
    \Lambda_g(\bm{h}):=\sigma(\bm{u},p)\nu \quad \text{on}\ \partial M.
\end{align}
It is clear that the Dirichlet-to-Neumann map $\Lambda_{g}$ is an elliptic, self-adjoint pseudodifferential operator of order one defined on the boundary $\partial M$. An interesting question is whether the Cauchy data $C_g$ (or the Dirichlet-to-Neumann map $\Lambda_g$) for the Stokes equations uniquely determines the geometry of the boundary of the manifold.

\vspace{2mm}

The main result of this paper is the following theorem.

\begin{theorem}\label{thm1.1}
    Let $(M,g)$ be a compact connected Riemannian manifold of dimension $n$ with smooth boundary $\partial M$, $n\geqslant 2$. Then, the Cauchy data $C_g$ (or the Dirichlet-to-Neumann map $\Lambda_g$) for the Stokes equations uniquely determines the partial derivatives of all orders of the metric $\frac{\partial^{|J|} g^{\alpha\beta}}{\partial x^J}$ on $\partial M$ for all multi-indices $J$.
\end{theorem}

The Dirichlet-to-Neumann maps have been studied for decades. In \cite{LeeUhlm89}, the authors proved that the Dirichlet-to-Neumann map uniquely determines the real-analytic Riemannian metric. In \cite{LassUhlm01}, the authors studied the inverse problem of determining a Riemannian manifold from the boundary data of harmonic functions, this extend the results in \cite{LeeUhlm89}. Moreover, \cite{LassTaylUhlm03} considered the case of complete Riemannian manifold. In \cite{Liu19,Liu19.2}, the author proved that the elastic Dirichlet-to-Neumann map with constant coefficients and the electromagnetic Dirichlet-to-Neumann map can uniquely determines the real-analytic Riemannian metric and parameters. In \cite{LiuTan23}, the authors computed the full symbol of the magnetic Dirichlet-to-Neumann map and calculated all the coefficients of the heat trace asymptotic expansion associated with the magnetic Steklov problem. In \cite{TanLiu23}, the authors gave an explicit expression for the full symbol of the elastic Dirichlet-to-Neumann map with variable coefficients and proved that the elastic Dirichlet-to-Neumann map uniquely determines the Lam\'{e} coefficients. In \cite{Tan23}, the author proved that the thermoelastic Dirichlet-to-Neumann map uniquely determines partial derivatives of
all orders of thermoelastic coefficients on the boundary of the manifold. We refer the reader to \cite{Uhlm14,Uhlm09} and the references therein for more topics about the Dirichlet-to-Neumann maps.

This paper is organized as follows. In Section \ref{s2}, we derive a new system associated with the Stokes equations. In Section \ref{s3}, we give the symbols of some pseudodifferential operators. In Section \ref{s4}, we prove the main result by the full symbol of the new Dirichlet-to-Neumann map.

\vspace{5mm}

\section{A new system associated with the Stokes equations}\label{s2}

\vspace{5mm}

In this section we will derive a new system associated with the Stokes equations \eqref{1.3}. Inspired by \cite{HeckLiWang07,Liu20}, we set 
\begin{align}\label{a1}
    \bm{u}=\mu^{-1/2}\bm{w}+\mu^{-1}\nabla f-f\nabla \mu^{-1}.
\end{align}
Then,
\begin{align}\label{a2}
    \operatorname{div}\bm{u}=\mu^{-1/2}\operatorname{div}\bm{w}+ g(\nabla\mu^{-1/2},\bm{w}) +\mu^{-1}\Delta_g f - f \Delta_g \mu^{-1}.
\end{align}
The $j$th component of $\operatorname{div}(\mu S\bm{u})$ is
\begin{align*}
    \operatorname{div}(\mu S\bm{u})^j
    &=\nabla^k\big(\mu (S\bm{u})^j_k\big)\\
    &=\nabla^k\big[-(\nabla^j\mu^{1/2})w_k-(\nabla_k\mu^{1/2})w^j+\mu^{1/2}(\nabla^j w_k+\nabla_k w^j)+2\nabla^j\nabla_k f\\
    &\quad -2\mu(\nabla^j\nabla_k \mu^{-1})f  \big]\\
    &=-(\nabla^k\nabla^j\mu^{1/2})w_k-(\nabla^j\mu^{1/2})\nabla^k w_k-(\nabla^k\nabla_k\mu^{1/2})w^j-(\nabla_k\mu^{1/2})\nabla^k w^j \\
    &\quad +(\nabla^k\mu^{1/2})(\nabla^j w_k+\nabla_k w^j)+\mu^{1/2}(\nabla^k\nabla^j w_k+\nabla^k\nabla_k w^j) +2\nabla^k\nabla^j\nabla_k f \\
    &\quad -2\nabla^k(\mu\nabla^j\nabla_k \mu^{-1})f-2\mu(\nabla^j\nabla_k \mu^{-1})\nabla^k f.
\end{align*}
Note that
\begin{align*}
    \nabla^k\nabla^j w_k
    &=g^{jl}\nabla_k\nabla_l w^k\\
    &=g^{jl}(\nabla_l\nabla_kw^k+R^k_{klm}w^m)\\
    &=g^{jl}(\nabla_l\nabla_kw^k+R_{lm}w^m)\\
    &=\nabla^j\operatorname{div}\bm{w}+\operatorname{Ric}(\bm{w})^j.
\end{align*}
Here $\operatorname{Ric}(\bm{w})^j = g^{jk}R_{kl} w^l$, where $R_{kl}$ are the components of Ricci tensor of the manifold, in local coordinates,
\begin{align}\label{0.03}
    R_{kl}=\frac{\partial \Gamma^{j}_{kl}}{\partial x_j} - \frac{\partial \Gamma^{j}_{jl}}{\partial x_k} + \Gamma^{j}_{jm} \Gamma^{m}_{kl} - \Gamma^{j}_{km} \Gamma^{m}_{jl},
\end{align}
where the Christoffel symbols
\begin{align*}
    \Gamma^{j}_{kl} = \frac{1}{2} g^{jm} \Bigl(\frac{\partial g_{km}}{\partial x_l} + \frac{\partial g_{lm}}{\partial x_k} - \frac{\partial g_{kl}}{\partial x_m}\Bigr).
\end{align*}

Similarly,
\begin{align*}
    \nabla^k\nabla^j\nabla_k f= \nabla^j\Delta_g f+\operatorname{Ric}(\nabla f)^j.
\end{align*}
Hence,
\begin{align*}
    \operatorname{div}(\mu S\bm{u})^j
    &=\nabla^j((\nabla^k\mu^{1/2})w_k)-2(\nabla^k\nabla^j\mu^{1/2})w_k+\nabla^j(\mu^{1/2}\operatorname{div}\bm{w})-2(\nabla^j\mu^{1/2})\operatorname{div}\bm{w} \\
    &\quad -(\Delta_g\mu^{1/2})w^j+\mu^{1/2}((\Delta^{}_{B}\bm{w})^j+\operatorname{Ric}(\bm{w})^j)+2(\nabla^j\Delta_g f+\operatorname{Ric}(\nabla f)^j) \\
    &\quad -2\nabla^k(\mu\nabla^j\nabla_k\mu^{-1})f-2\mu(\nabla^j\nabla_k\mu^{-1})\nabla^k f,
\end{align*}
where the Bochner Laplacian is given by $(\Delta^{}_{B}\bm{w})^j := \nabla^k \nabla_k w^j$. Let
\begin{align*}
    p=\operatorname{div}(\mu^{1/2}\bm{w})+2\Delta_g f.
\end{align*}
Then we have
\begin{align*}
    \operatorname{div}(\mu S\bm{u})^j
    &=\mu^{1/2}((\Delta^{}_{B}\bm{w})^j+\operatorname{Ric}(\bm{w})^j)+\nabla^j p -2(\nabla^j\mu^{1/2})\operatorname{div}\bm{w}-2\nabla^k(\mu\nabla^j\nabla_k\mu^{-1})f \\
    &\quad -2\mu(\nabla^j\nabla_k\mu^{-1})\nabla^k f-2(\nabla^k\nabla^j\mu^{1/2})w_k-(\Delta_g\mu^{1/2})w^j+2\operatorname{Ric}(\nabla f)^j.
\end{align*}
Therefore,
\begin{align*}
    (\operatorname{div}\sigma)^j
    &=\operatorname{div}(\mu S\bm{u})^j-\nabla^j p\\
    &=\mu^{1/2}((\Delta^{}_{B}\bm{w})^j+\operatorname{Ric}(\bm{w})^j) -2(\nabla^j\mu^{1/2})\operatorname{div}\bm{w}-2\nabla^k(\mu\nabla^j\nabla_k\mu^{-1})f \\
    &\quad -2\mu(\nabla^j\nabla_k\mu^{-1})\nabla^k f-2(\nabla^k\nabla^j\mu^{1/2})w_k-(\Delta_g\mu^{1/2})w^j+2\operatorname{Ric}(\nabla f)^j.
\end{align*}
It follows from \cite{TanLiu23} that the Bochner Laplacian can be written as
\begin{align*}
    (\Delta^{}_{B}\bm{w})^j=\Delta_g w^j - \operatorname{Ric}(\bm{w})^j + g^{kl} \Bigl( 2\Gamma^j_{mk} \frac{\partial w^m}{\partial x_l} + \frac{\partial \Gamma^j_{kl}}{\partial x_m} w^m \Bigr),
\end{align*}
and the divergence operator has the local expression
\begin{align*}
    \operatorname{div}\bm{w} = \frac{\partial w^k}{\partial x_k} + \Gamma^k_{kl} w^l.
\end{align*}
Hence, we get
\begin{align}\label{b3}
    (\operatorname{div}\sigma)^j
    &=\mu^{1/2}\biggl(\Delta_g w^j + g^{kl} \Bigl( 2\Gamma^j_{mk} \frac{\partial w^m}{\partial x_l} + \frac{\partial \Gamma^j_{kl}}{\partial x_m} w^m \Bigr)\biggr)  \notag\\
    &\quad -2(\nabla^j\mu^{1/2})\Bigl(\frac{\partial w^k}{\partial x_k} + \Gamma^k_{kl} w^l\Bigr)-2\nabla^k(\mu\nabla^j\nabla_k\mu^{-1})f \notag\\
    &\quad -2\mu(\nabla^j\nabla_k\mu^{-1})g^{kl}\frac{\partial f}{\partial x_l} -2(\nabla^k\nabla^j\mu^{1/2})w_k-(\Delta_g\mu^{1/2})w^j+2R^{jk}\frac{\partial f}{\partial x_k}.
\end{align}

In boundary normal coordinates, the metric has the form (see \cite{LeeUhlm89,TanLiu23,LiuTan23,Tan23})
\begin{align*}
    g = g_{\alpha\beta} \,dx_{\alpha} \,dx_{\beta} + dx_{n}^{2}.
\end{align*}
Note that in this coordinates, in a neighborhood of the origin, we have
\begin{align*}
    g_{\alpha n}&=g^{\alpha n}=0,\\
    \Gamma^{n}_{nk}&=\Gamma^{k}_{nn}=0.
\end{align*}
Then, in boundary normal coordinates, we write the Laplace--Beltrami operator as
\begin{align}\label{b2}
    \Delta_g
    & = \frac{\partial^2 }{\partial x_n^2} + \Gamma^\alpha_{\alpha n} \frac{\partial }{\partial x_n} + g^{\alpha\beta} \frac{\partial^2}{\partial x_\alpha\partial x_\beta} +
    \Bigl(
        g^{\alpha\beta} \Gamma^\gamma_{\gamma\alpha} + \frac{\partial g^{\alpha\beta}}{\partial x_\alpha}
    \Bigr)
    \frac{\partial }{\partial x_\beta}.
\end{align}

In what follows, for the sake of simplicity, we denote by $R^{jk}=g^{jl}g^{km}R_{lm}$, $I_n$ the $n\times n$ identity matrix, and
\begin{align*}
    \begin{bmatrix}
        [a^j_k] & [b^j] \\[2mm]
        [c_k] & d
    \end{bmatrix}
    :=
    \begin{bmatrix}
        \begin{BMAT}{ccc.c}{ccc.c}
        a^1_1 & \dots & a^1_n\ &\ b^1 \\
        \vdots & \ddots & \vdots\ &\ \vdots \\
        a^n_1 & \dots & a^n_n\ &\ b^n \\
        c_1 & \dots & c_n\ &\ d
    \end{BMAT}
    \end{bmatrix}.
\end{align*}

Let $\bm{U}=(\bm{w},f)^T$. Combining \eqref{1.3}, \eqref{a2}, \eqref{b3}, and \eqref{b2}, we obtain, in boundary normal coordinates,
\begin{align}\label{b8}
    L_g \bm{U} = 0.
\end{align}
Here the operator $L_g$ is given by the following equality
\begin{align}\label{3.07}
    A^{-1} L_g = I_{n+1}\frac{\partial^2 }{\partial x_n^2} + B \frac{\partial }{\partial x_n} + C,
\end{align}
where
\begin{align}\label{3.2}
        A=
        \begin{bmatrix}
            \mu^{1/2} I_{n} &0  \\
            0& \mu^{-1} 
        \end{bmatrix},
\end{align}
\begin{align*}
    &B=\Gamma^\alpha_{\alpha n}I_{n+1}+
    \begin{bmatrix}
        2[\Gamma^j_{kn}]-2\mu^{-1/2}[(\nabla^j\mu^{1/2})\delta_{nk}] & 2\mu^{-1/2}[R^{jn}-\mu\nabla^j\nabla^n\mu^{-1}]\\
        \mu^{1/2}[\delta_{nk}] & 0 
    \end{bmatrix},\\
    &C=C_2+C_1+C_0,\\
    &C_2=\Big(g^{\alpha\beta}\frac{\partial^2 }{\partial x_\alpha \partial x_\beta}\Big)I_{n+1},\\
    &C_1=\bigg(\Big(g^{\alpha\beta}\Gamma^\gamma_{\alpha\gamma}+\frac{\partial g^{\alpha\beta}}{\partial x_{\alpha}}\Big) \frac{\partial }{\partial x_{\beta}}\bigg) I_{n+1}\\
    &\qquad +\begin{bmatrix}
        \displaystyle 2\Big[g^{\alpha\beta}\Gamma^j_{k\alpha}\frac{\partial}{\partial x_{\beta}}\Big]  & \displaystyle 2\mu^{-1/2}\Bigl[(R^{j\alpha}-\mu\nabla^j\nabla^\alpha\mu^{-1})\frac{\partial}{\partial x_{\alpha}}\Bigr] \\
        \displaystyle \mu^{1/2}\Bigl[\frac{\partial}{\partial x_{k}}-\delta_{nk}\frac{\partial}{\partial x_{n}}\Bigr] &  0
    \end{bmatrix}\\
    &\qquad +
    \begin{bmatrix}
        \displaystyle -2\mu^{-1/2}\Big[(\nabla^j\mu^{1/2})\Bigl(\frac{\partial}{\partial x_{k}}-\delta_{nk}\frac{\partial}{\partial x_{n}}\Bigr)\Big]  & 0 \\
        0 &  0
    \end{bmatrix},\\
    &C_0 =
    \begin{bmatrix}
        \displaystyle (-\mu^{-1/2}\Delta_g \mu^{1/2})I_n + \bigg[g^{ml}\frac{\partial \Gamma^j_{ml}}{\partial x_k}\bigg]  &  \displaystyle -2\mu^{-1/2}[\nabla^k(\mu\nabla^j\nabla_k\mu^{-1})] \\[4mm]
        \displaystyle \mu^{1/2}[\Gamma^l_{lk}]+\mu[\nabla_k\mu^{-1/2}] & -\mu\Delta_g\mu^{-1}
    \end{bmatrix}\\
    &\qquad +
    \begin{bmatrix}
        \displaystyle - 2\mu^{-1/2}[(\nabla^j\mu^{1/2})\Gamma^l_{lk}] - 2\mu^{-1/2}[\nabla^j\nabla_k\mu^{1/2}] &  \displaystyle 0 \\[4mm]
        \displaystyle 0 &  0
    \end{bmatrix}.
\end{align*}

We define the Dirichlet-to-Neumann map $\tilde{\Lambda}_g:[H^{3/2}(\partial M)]^{n+1}\to [H^{1/2}(\partial M)]^{n+1}$ associated with the following Dirichlet problem
\begin{align}\label{b9}
    \begin{cases}
        L_g\bm{U}=0 & \text{in}\ M,\\
        \bm{U}=\bm{V} \quad & \text{on}\ \partial M
    \end{cases}
\end{align}
by 
\begin{align}
    \tilde{\Lambda}_g(\bm{V}):=\frac{\partial \bm{U}}{\partial \nu} \quad \text{on}\ \partial M.
\end{align}
The corresponding Cauchy data is given by
\begin{align}
    \tilde{C}_g:=\Big\{\Big(\bm{U},\frac{\partial \bm{U}}{\partial \nu}\Big)\Big|_{\partial M}:\bm{U}\ \text{satisfies \eqref{b9}}\Big\}.
\end{align}
It is clear that the Cauchy data $\tilde{C}_g$ corresponding to the Dirichlet-to-Neumann map $\tilde{\Lambda}_g$ is equivalent to the Cauchy data $C_g$ corresponding to the Dirichlet-to-Neumann map $\Lambda_g$.

\vspace{5mm}

\section{Symbols of the pseudodifferential operators}\label{s3}

\vspace{5mm}

We denote by $i=\sqrt{-1}$, $\xi^{\prime}=(\xi_1,\dots,\xi_{n-1})$, $\xi^\alpha=g^{\alpha\beta}\xi_\beta$, $|\xi^{\prime}|=\sqrt{\xi^\alpha\xi_\alpha}$. Let $b(x,\xi^{\prime})$ and
    \begin{align*}
        c(x,\xi^{\prime}) = c_2(x,\xi^{\prime}) + c_1(x,\xi^{\prime}) + c_0(x,\xi^{\prime})
    \end{align*}
    be the full symbols of $B$ and $C$, respectively, where $c_j(x,\xi^{\prime})$ are homogeneous of degree $j$ in $\xi^{\prime}$. Thus, we obtain
    \begin{align}
        &\label{5.10} b(x,\xi^{\prime})=B,\\
        &\label{5.11} c_2(x,\xi^{\prime})= -|\xi^{\prime}|^2 I_{n+1},\\
        &\label{5.12} c_1(x,\xi^{\prime})= i\Big(\xi^\alpha\Gamma^\beta_{\alpha\beta}+\frac{\partial \xi^\alpha}{\partial x_{\alpha}}\Big) I_{n+1} \notag\\
        &\qquad+ i
        \begin{bmatrix}
            \displaystyle 2\big[g^{\alpha\beta}\Gamma^j_{k\alpha}\xi_{\beta}\big]  & \displaystyle 2\mu^{-1/2}\bigl[(R^{j\alpha}-\mu\nabla^j\nabla^\alpha\mu^{-1})\xi_{\alpha}\bigr] \notag\\
            \displaystyle \mu^{1/2}\bigl[\xi_{k}-\delta_{nk}\xi_{n}\bigr] &  0
        \end{bmatrix}\\
        &\qquad + i
        \begin{bmatrix}
            \displaystyle -2\mu^{-1/2}\big[(\nabla^j\mu^{1/2})(\xi_{k}-\delta_{nk}\xi_{n})\big]  & 0 \\
            0 &  0
        \end{bmatrix},\\
        &\label{5.13} c_0(x,\xi^{\prime}) =C_0.
    \end{align} 

For the convenience of stating the following proposition, we define 
\begin{align}
    \label{2.11} E_1&:=i\sum_\alpha\frac{\partial q_1}{\partial \xi_\alpha}\frac{\partial q_1}{\partial x_\alpha}+bq_1+\frac{\partial q_1}{\partial x_n} - c_1,\\
    \label{3.05} E_0&:=i\sum_\alpha\Bigl(\frac{\partial q_1}{\partial \xi_\alpha}\frac{\partial q_0}{\partial x_\alpha}+\frac{\partial q_0}{\partial \xi_\alpha}\frac{\partial q_1}{\partial x_\alpha}\Bigr)+\frac{1}{2}\sum_{\alpha,\beta}\frac{\partial^2q_1}{\partial \xi_\alpha \partial\xi_\beta}\frac{\partial^2q_1}{\partial x_\alpha \partial x_\beta} \notag\\
    &\quad -q_0^2 +bq_0 +\frac{\partial q_0}{\partial x_n} - c_0,\\
    \label{4.1} E_{-m}&:= bq_{-m}+\frac{\partial q_{-m}}{\partial x_n} - \sum_{\substack{-m \leqslant j,k \leqslant 1 \\ |J| = j + k + m}} \frac{(-i)^{|J|}}{J !} \partial_{\xi^{\prime}}^{J} q_j\, \partial_{x^\prime}^{J} q_k,\quad m \geqslant 1,
\end{align}
where $q_j=q_j(x,\xi^{\prime})$, $b=b(x,\xi^{\prime})$, and $c_j=c_j(x,\xi^{\prime})$.

\begin{proposition}\label{prop3.1}
    Let $Q(x,\partial_{x^\prime})$ be a pseudodifferential operator of order one in $x^\prime$ depending smoothly on $x_n$ such that
    \begin{align*}
        A^{-1}L_g
        = \Bigl(I_{n+1}\frac{\partial }{\partial x_n} + B - Q\Bigr)\Bigl(I_{n+1}\frac{\partial }{\partial x_n} + Q\Bigr)
    \end{align*}
    modulo a smoothing operator. Let $q(x,\xi^{\prime}) \sim \sum_{j\leqslant 1} q_j(x,\xi^{\prime})$ be the full symbol of $Q$, where $q_j(x,\xi^{\prime})$ are homogeneous of degree $j$ in $\xi^{\prime}$. Then, in boundary normal coordinates,
    \begin{align}
        q_1(x,\xi^{\prime})&=|\xi^{\prime}|I_{n+1}, \label{3.9}\\
        q_{-m-1}(x,\xi^{\prime})&=\frac{1}{2|\xi^{\prime}|}E_{-m}, \quad m\geqslant -1, \label{3.1.1}
    \end{align}
    where $E_{-m}\,(m\geqslant -1)$ are given by \eqref{2.11}--\eqref{4.1}.
\end{proposition}

\begin{proof}
    It follows from \eqref{3.07} that
    \begin{align*}
        I_{n+1}\frac{\partial^2 }{\partial x_n^2} + B \frac{\partial }{\partial x_n} + C 
        = \Bigl(I_{n+1}\frac{\partial }{\partial x_n} + B - Q\Bigr)\Bigl(I_{n+1}\frac{\partial }{\partial x_n} + Q\Bigr)
    \end{align*}
    modulo a smoothing operator. Equivalently,
    \begin{align}\label{3.7}
        Q^2 - BQ - \Big[I_{n+1}\frac{\partial }{\partial x_n},Q\Big] + C = 0
    \end{align}
    modulo a smoothing operator, where the commutator $\big[I_{n+1}\frac{\partial }{\partial x_n},Q\big]$ is defined by, for any $v\in C^{\infty}(M)$,
    \begin{align*}
        \Big[I_{n+1}\frac{\partial }{\partial x_n},Q\Big]v
        &:= I_{n+1}\frac{\partial }{\partial x_n}(Qv) - Q \Bigl(I_{n+1}\frac{\partial }{\partial x_n}\Bigr)v \\
        &= \frac{\partial Q}{\partial x_n}v.
    \end{align*}
    Recall that if $G_1$ and $G_2$ are two pseudodifferential operators with full symbols $g_1=g_1(x,\xi)$ and $g_1=g_2(x,\xi)$, respectively, then the full symbol $\sigma(G_1G_2)$ of the operator $G_1G_2$ is given by (see \cite[p.\,11]{Taylor11.2}, \cite[p.\,71]{Hormander85.3}, and also \cite{Grubb86,Treves80})
    \begin{align*}
        \sigma(G_1G_2)\sim \sum_{J} \frac{(-i)^{|J|}}{J !} \partial_{\xi}^{J}g_1 \, \partial_{x}^{J}g_2,
    \end{align*}
    where the sum is over all multi-indices $J$. Let $q = q(x,\xi^{\prime})$ be the full symbol of the operator $Q(x,\partial_{x^\prime})$, we write $ q(x,\xi^{\prime}) \sim \sum_{j\leqslant 1} q_j(x,\xi^{\prime})$ with $q_j(x,\xi^{\prime})$ homogeneous of degree $j$ in $\xi^{\prime}$. Hence, we get the following full symbol equation of \eqref{3.7}
    \begin{equation}\label{3.8}
        \sum_{J} \frac{(-i)^{|J|}}{J !} \partial_{\xi^{\prime}}^{J}q \, \partial_{x^\prime}^{J}q - \sum_{J} \frac{(-i)^{|J|}}{J !} \partial_{\xi^{\prime}}^{J}b \, \partial_{x^\prime}^{J}q - \frac{\partial q}{\partial x_n} + c = 0,
    \end{equation}
    where the sum is over all multi-indices $J$.

    \vspace{2mm}

    We shall determine $q_j=q_j(x,\xi^{\prime})\,(j\leqslant 1)$ so that \eqref{3.8} holds modulo $S^{-\infty}$. Grouping the homogeneous terms of degree two in \eqref{3.8}, we have
    \begin{align}\label{5.1}
        q_1^2+c_2=0.
    \end{align}
    Since we have chosen the unit outer normal vector $\nu$ on the boundary, by combining the above equation and \eqref{5.11}, we take
    \begin{align}\label{5.15}
        q_1=|\xi^{\prime}|I_{n+1},
    \end{align}
    which implies that $q_1$ is positive definite.

    Grouping the homogeneous terms of degree $-m\,(m\geqslant -1)$ in \eqref{3.8}, we get
    \begin{align}\label{4.2}
        q_1q_{-m-1}+q_{-m-1}q_1=E_{-m},
    \end{align}
    where $E_{-m}\,(m\geqslant -1)$ are given by \eqref{2.11}--\eqref{4.1}. By \eqref{5.15} and \eqref{4.2} we immediately get
\begin{align*}
    q_{-m-1}(x,\xi^{\prime})&=\frac{1}{2|\xi^{\prime}|}E_{-m}.
\end{align*}
\end{proof}

In boundary normal coordinates, the Dirichlet-to-Neumann map $\tilde{\Lambda}_{g}$ can be represented as the pseudodifferential operator $Q$ modulo a smoothing operator (see the following Proposition \ref{prop3.3}).
\begin{proposition}\label{prop3.3}
    In boundary normal coordinates, the Dirichlet-to-Neumann map $\tilde{\Lambda}_{g}$ can be represented as
    \begin{align}\label{3.10}
        \tilde{\Lambda}_{g}\bm{U} = Q\bm{U}|_{\partial M}
    \end{align}
    modulo a smoothing operator.
\end{proposition}
\begin{proof}
    We use the boundary normal coordinates $(x^\prime,x_n)$ with $x_n\in[0,T]$. Since the principal symbol of the operator $L_{g}$ is negative definite, the hyperplane ${x_n = 0}$ is non-characteristic. Hence, $L_{g}$ is partially hypoelliptic with respect to this boundary (see \cite[p.\,107]{Hormander64}). Therefore, the solution to the equation $L_{g} \bm{U} = 0$ is smooth in normal variable, that is, $\bm{U}\in [C^\infty([0,T];\mathfrak{D}^\prime (\mathbb{R}^{n-1}))]^{n+1}$ locally. From Proposition \ref{prop3.1}, we see that \eqref{b9} is locally equivalent to the following system of equations for $\bm{U},\bm{W}\in [C^\infty([0,T];\mathfrak{D}^\prime (\mathbb{R}^{n-1}))]^{n+1}$:
    \begin{align*}
        \Big(I_{n+1}\frac{\partial }{\partial x_n} + Q\Big)\bm{U}  &= \bm{W}, \quad \bm{U}|_{x_n=0}=\bm{V}, \\
        \Big(I_{n+1}\frac{\partial }{\partial x_n} + B - Q\Big)\bm{W}  &=\bm{Y} \in [C^\infty([0,T]\times \mathbb{R}^{n-1})]^{n+1}.
    \end{align*}
    Inspired by \cite{LeeUhlm89,Liu19,TanLiu23,Tan23}, if we substitute $t = T-x_n$ into the second equation above, then we get a backwards generalized heat equation
    \begin{align*}
        \frac{\partial \bm{W}}{\partial t} -(B-Q)\bm{W}=-\bm{Y}.
    \end{align*}
    Since $\bm{U}$ is smooth in the interior of the manifold $M$ by interior regularity for elliptic operator $L_{g}$, it follows that $\bm{W}$ is also smooth in the interior of $M$, and so $\bm{W}|_{x_n=T}$ is smooth. In view of that $q_1$ (the principal symbol of $Q$) is positive definite (see \eqref{3.9}), we get that the solution operator for this heat equation is smooth for $t > 0$ (see \cite[p.\,134]{Treves80}). Therefore,
    \begin{align*}
        \frac{\partial \bm{U}}{\partial x_n} + Q\bm{U}=\bm{W} \in [C^\infty([0,T]\times \mathbb{R}^{n-1})]^{n+1}
    \end{align*}
    locally. If we set $\mathcal{R} \bm{V} = \bm{W}|_{\partial M}$, this shows that $\mathcal{R}$ is a smoothing operator and
    \begin{align}\label{5.8}
        \frac{\partial \bm{U}}{\partial x_n}\bigg|_{\partial M}=-Q\bm{U}|_{\partial M}+\mathcal{R}\bm{V}.
    \end{align}
\end{proof}

\vspace{5mm}

\section{Determining the metric on the boundary}\label{s4}

\vspace{5mm}

\begin{proof}[Proof of Theorem {\rm \ref{thm1.1}}]
    Since the Cauchy data $\tilde{C}_g$ is equivalent to the Cauchy data $C_g$, it suffices to show that the Dirichlet-to-Neumann map $\tilde{\Lambda}_{g}$ (or the pseudodifferential operator $Q$) uniquely determines the on the metric boundary by Proposition \ref{prop3.3}.

It follows from \eqref{3.9} that
\begin{align*}
    q_1(x,\xi^{\prime})=|\xi^{\prime}|I_{n+1}=\sqrt{g^{\alpha\beta}\xi_\alpha\xi_\beta}I_{n+1}.
\end{align*}
This shows that $q_1$ uniquely determines $g^{\alpha\beta}|_{\partial M}$ for all $1\leqslant\alpha,\beta\leqslant n-1$. Clearly, the tangential derivatives $\frac{\partial g^{\alpha\beta}}{\partial x_\gamma}\big|_{\partial M}$ can also be uniquely determined by $q_1$ for all $1\leqslant \alpha,\beta,\gamma \leqslant n-1$.

For $k\geqslant 0$, we denote by $T_{-k}$ the terms that only involve the boundary values of $g_{\alpha\beta}$, $g^{\alpha\beta}$, and their normal derivatives of order ar most $k$. Note that $T_{-k}$ may be different in different expressions. From \eqref{5.10}, \eqref{5.12}, \eqref{2.11}, and \eqref{3.9}, we know that
\begin{align*}
    E_1&=bq_1+\frac{\partial q_1}{\partial x_n}+T_0,\\
    \operatorname{tr}E_1&=(n+3)\Gamma^\alpha_{\alpha n}|\xi^{\prime}|+(n+1)\frac{\partial |\xi^{\prime}|}{\partial x_n}+T_0.
\end{align*}
By \eqref{3.1.1}, we get
\begin{align}\label{b1}
    \operatorname{tr} q_0 =\frac{1}{2}\Bigl((n+3)\Gamma^\alpha_{\alpha n}+(n+1)\frac{1}{|\xi^{\prime}|}\frac{\partial |\xi^{\prime}|}{\partial x_n}\Bigr)+T_0.
\end{align}
In boundary normal coordinates, we have
\begin{align*}
    \Gamma^{\alpha}_{n\alpha}=\frac{1}{2}g^{\alpha\beta}\frac{\partial g_{\alpha\beta}}{\partial x_n}=-\frac{1}{2}g_{\alpha\beta}\frac{\partial g^{\alpha\beta}}{\partial x_n}.
\end{align*}
Substituting this into \eqref{b1}, we get
\begin{align}
    \label{6.8} \operatorname{tr} q_0
    &=-\frac{1}{4}\Bigl((n+3)g_{\alpha\beta}\frac{\partial g^{\alpha\beta}}{\partial x_n} - (n+1)\frac{1}{|\xi^{\prime}|^2}\frac{\partial |\xi^{\prime}|^2}{\partial x_n}\Bigr)+T_0 \notag\\
    &=-\frac{1}{4|\xi^{\prime}|^2}k_1^{\alpha\beta}\xi_\alpha\xi_\beta+T_0,
\end{align}
where
\begin{align}
    \label{6.9} k_1^{\alpha\beta}&=(n+3)h_1g^{\alpha\beta} - (n+1)\frac{\partial g^{\alpha\beta}}{\partial x_n},\\
    \label{6.10} h_1&=g_{\alpha\beta}\frac{\partial g^{\alpha\beta}}{\partial x_n}.
\end{align}
Evaluating $\operatorname{tr} q_0$ on all unit vectors $\xi^{\prime}$ shows that $q_0$ and $g^{\alpha\beta}|_{\partial M}$ completely determine $k_1^{\alpha\beta}$. By \eqref{6.9} and \eqref{6.10}, we have 
\begin{align*}
    k_1^{\alpha\beta}g_{\alpha\beta}=(n^2+n-4)h_1.
\end{align*}
For $n\geqslant 2$, we have $n^2+n-4>0$. Hence,
\begin{align*}
    h_1=\frac{k_1^{\alpha\beta}g_{\alpha\beta}}{n^2+n-4}.
\end{align*}
By \eqref{6.9}, we get that
\begin{align*}
    \frac{\partial g^{\alpha\beta}}{\partial x_n}=\frac{(n+3)h_1g^{\alpha\beta}-k_1^{\alpha\beta}}{n+1},
\end{align*}
which implies that $q_0$ uniquely determines $\frac{\partial g^{\alpha\beta}}{\partial x_n}\big|_{\partial M}$.

It follows from \eqref{3.05} that
\begin{align*}
    E_0=\frac{\partial q_0}{\partial x_n} - c_0+T_{-1}.
\end{align*}
By \eqref{b1} and \eqref{5.13}, we obtain
\begin{align}\label{b7}
    \operatorname{tr}E_0
    =\frac{1}{2}\Bigl((n+3)\frac{\partial \Gamma^\alpha_{\alpha n}}{\partial x_n}+(n+1)\frac{1}{|\xi^{\prime}|}\frac{\partial^2 |\xi^{\prime}|}{\partial x_n^2}\Bigr)-g^{ml}\frac{\partial \Gamma^j_{ml}}{\partial x_j}  +T_{-1}.
\end{align}

Note that
\begin{align}\label{b4}
    \frac{1}{|\xi^{\prime}|}\frac{\partial^2 |\xi^{\prime}|}{\partial x_n^2}
    =\frac{1}{2|\xi^{\prime}|^2}\frac{\partial^2 |\xi^{\prime}|^2}{\partial x_n^2}+T_{-1}.
\end{align}
In view of that
\begin{align*}
    \frac{\partial^2 (g_{\alpha\beta}g^{\alpha\beta})}{\partial x_n^2}=\frac{\partial^2 (n-1)}{\partial x_n^2}=0,
\end{align*}
we get
\begin{align*}
    g_{\alpha\beta}\frac{\partial^2 g^{\alpha\beta}}{\partial x_n^2} = - g^{\alpha\beta}\frac{\partial^2 g_{\alpha\beta}}{\partial x_n^2}+T_{-1}.
\end{align*}
Then, in boundary normal coordinates, we compute that
\begin{align}
    \frac{\partial \Gamma^{\alpha}_{n\alpha}}{\partial x_n}&=-\frac{1}{2}g_{\alpha\beta}\frac{\partial^2 g^{\alpha\beta}}{\partial x_n^2}+T_{-1},\label{b5}\\
    g^{ml}\frac{\partial \Gamma^j_{ml}}{\partial x_j}&=\frac{1}{2}g_{\alpha\beta}\frac{\partial^2 g^{\alpha\beta}}{\partial x_n^2}+T_{-1}.\label{b6}
\end{align}
Combining \eqref{3.1.1}, \eqref{b7}, \eqref{b5}, and \eqref{b6}, we have 
\begin{align}
    \label{6.12} \operatorname{tr}q_{-1}
    &=-\frac{n+5}{8|\xi^{\prime}|}g_{\alpha\beta}\frac{\partial^2 g^{\alpha\beta}}{\partial x_n^2}+\frac{n+1}{8|\xi^{\prime}|^3}\frac{\partial^2 |\xi^{\prime}|^2}{\partial x_n^2}+T_{-1} \notag\\
    &=-\frac{1}{8|\xi^{\prime}|^3}k_2^{\alpha\beta}\xi_\alpha\xi_\beta+T_{-1},
\end{align}
where
\begin{align}
    \label{6.13} k_2^{\alpha\beta}&=(n+5)h_2g^{\alpha\beta}-(n+1)\frac{\partial^2 g^{\alpha\beta}}{\partial x_n^2},\\
    \label{6.14} h_2&=g_{\alpha\beta}\frac{\partial^2 g^{\alpha\beta}}{\partial x_n^2}.
\end{align}
By the same argument, it follows from \eqref{6.13} and \eqref{6.14} that
\begin{align*}
    k_2^{\alpha\beta}g_{\alpha\beta}=(n^2+3n-6)h_2.
\end{align*}
For $n\geqslant 2$, we have $n^2+3n-6>0$. Hence,
\begin{align*}
    h_2=\frac{k_2^{\alpha\beta}g_{\alpha\beta}}{n^2+3n-6}.
\end{align*}
By \eqref{6.13}, we get that
\begin{align*}
    \frac{\partial^2 g^{\alpha\beta}}{\partial x_n^2}=\frac{(n+5)h_2g^{\alpha\beta}-k_2^{\alpha\beta}}{n+1},
\end{align*}
which implies that $q_{-1}$ uniquely determines $\frac{\partial^2 g^{\alpha\beta}}{\partial x_n^2}\big|_{\partial M}$.

Now we consider $q_{-m-1}$ for $m\geqslant 1$. From \eqref{4.1}, we see that
\begin{align}
    \label{6.18} E_{-m}=\frac{\partial q_{-m}}{\partial x_n}+T_{-m-1}.
\end{align}
We end this proof by induction. Suppose we have shown that
\begin{align}
    \label{6.15} \operatorname{tr}E_{-j}
    &= -\frac{1}{(2|\xi^{\prime}|)^{j+2}}k_{j+2}^{\alpha\beta}\xi_\alpha\xi_\beta+T_{-j-1}
\end{align}
for $1\leqslant j \leqslant m$, where
\begin{align}
    \label{6.16} k_{j+2}^{\alpha\beta}&=(n+5)h_{j+2}g^{\alpha\beta}-(n+1)\frac{\partial^{j+2} g^{\alpha\beta}}{\partial x_n^{j+2}},\\
    \label{6.17} h_{j+2}&=g_{\alpha\beta}\frac{\partial^{j+2} g^{\alpha\beta}}{\partial x_n^{j+2}}.
\end{align}
This means that $q_{-j-1}$ uniquely determines $\frac{\partial^{j+2} g^{\alpha\beta}}{\partial x_n^{j+2}}\big|_{\partial M}$ for $1\leqslant j \leqslant m$.

Since we have $q_{-(m+1)-1}$ uniquely determines $E_{-(m+1)}$. From \eqref{6.18}, we have
\begin{align*}
    E_{-(m+1)}=\frac{\partial q_{-(m+1)}}{\partial x_n}+T_{-(m+1)-1}.
\end{align*}
By the above equality and the fact that $q_{-(m+1)}$ uniquely determines $E_{-m}$, we have $E_{-(m+1)}$ uniquely determines $\frac{\partial E_{-m}}{\partial x_n}$. By the assumption \eqref{6.15}, we get 
\begin{align*}
    \operatorname{tr} E_{-m-1}
    &=\frac{1}{2|\xi^{\prime}|}\frac{\partial (\operatorname{tr} E_{-m})}{\partial x_n}+T_{-m-2}\\
    &= -\frac{1}{(2|\xi^{\prime}|)^{m+3}}k_{m+3}^{\alpha\beta}\xi_\alpha\xi_\beta+T_{-m-2},
\end{align*}
where
\begin{align*}
    k_{m+3}^{\alpha\beta}&=(n+5)h_{m+3}g^{\alpha\beta}-(n+1)\frac{\partial^{m+3} g^{\alpha\beta}}{\partial x_n^{m+3}},\\
    h_{m+3}&=g_{\alpha\beta}\frac{\partial^{m+3} g^{\alpha\beta}}{\partial x_n^{m+3}}.
\end{align*}
By the same argument, we see that $q_{-(m+1)-1}$ uniquely determines $\frac{\partial^{m+3} g^{\alpha\beta}}{\partial x_n^{m+3}}\big|_{\partial M}$. Therefore, we conclude that the Dirichlet-to-Neumann map $\tilde{\Lambda}_{g}$ uniquely determines the partial derivatives of all orders of the Riemannian metric $\frac{\partial^{|J|} g^{\alpha\beta}}{\partial x^J}$ on the boundary $\partial M$ for all multi-indices $J$.

\end{proof}

\vspace{5mm}

\section*{Acknowledgements}

\vspace{5mm}

This work was supported by National Key R\&D Program of China 2020YFA0712800.

\vspace{5mm}

\end{document}